\documentclass{amsart}
\usepackage{graphicx,,amssymb}
\usepackage{amsfonts}
\usepackage{mathrsfs}
\usepackage[all]{xy}
\usepackage{overpic}
\newtheorem{thm}{Theorem}[section]
\newtheorem{lemma}[thm]{Lemma}
\newtheorem{prop}[thm]{Proposition}
\newtheorem{cor}[thm]{Corollary}

\theoremstyle{definition}
\newtheorem{defi}[thm]{Definition}
\newtheorem{example}[thm]{Example}

\theoremstyle{remark}
\newtheorem{remark}[thm]{Remark}
\numberwithin{equation}{section}

\newcommand{\<}{\langle}
\renewcommand{\>}{\rangle}
\newcommand{\st}{\; | \;}             %%  such that

\newcommand{\PP}{\mathbb{P}}      % P - proj. space
\newcommand{\OO}{\mathcal{O}}     % Structure sheaf
\newcommand{\F}{\mathcal{F}}      % sheaf
\newcommand{\D}{\mathcal{D}}      % derived category
\newcommand{\A}{\mathcal{A}}      % category
\newcommand{\B}{\mathcal{B}}      % category
\newcommand{\C}{\mathcal{C}}      % category
\newcommand{\Gammahat}{{\widehat{\Gamma}}}
    
\newcommand{\Hall}{\mathcal{H}}
\newcommand{\hh}{{\bf{h}}}

\newcommand{\Om}{\Omega}

\renewcommand{\th}{\theta}

\newcommand{\Z}{\mathbb{Z}}       % integers
\newcommand{\K}{\mathcal{K}}
\newcommand{\N}{\mathbb{N}}

\DeclareMathOperator{\Ind}{Ind}
\DeclareMathOperator{\Hom}{Hom}
\DeclareMathOperator{\RHom}{RHom}
\DeclareMathOperator{\Ext}{Ext}

\DeclareMathOperator{\id}{id}
\DeclareMathOperator{\Path}{Path}
\DeclareMathOperator{\Rep}{Rep}

\begin{document}
\title{Categorical Construction of $A,D,E$ Root Systems}
\author{A. Kirillov Jr.}
 \address{Department of Mathematics, SUNY at Stony Brook, 
            Stony Brook, NY 11794, USA}
    \email{kirillov@math.sunysb.edu}
    \urladdr{http://www.math.sunysb.edu/\textasciitilde kirillov/} 
\author{J. Thind}
 \address{Department of Mathematics and Statistics, Queen's University, 
            Kingston, ON, Canada}
    \email{jthind@mast.queensu.ca}
\maketitle
%%%%%%%%%%%%%%%%%%%%%%%%%%%%%%%%%%%%%%%%%%%%%%%%%%%%%%%%%%

\begin{abstract}
Let $\Gamma$ be a Dynkin diagram of type $A,D,E$ and let $R$ denote the corresponding root system. In this paper we give a categorical construction of $R$ from $\Gamma$. Instead of choosing an orientation of $\Gamma$ and studying representations of the associated quiver, we study representations of a canonical quiver $\Gammahat$ associated to $\Gamma$. This construction is very closely related to the preprojective algebra of $\Gamma$. In particular, the construction gives a certain periodicity result about the preprojective algebra.
\end{abstract}

%%%%%%%%%%%%%%%%%%%%%%%%%%%%%%%%%%%%%%%%%%%%%%%%%%%%%%%%%%
\section{Introduction}\label{s:intro}
In the 1970's Gabriel showed that when the underlying graph of a quiver $\overrightarrow{\Gamma}$ is a Dynkin diagram of type $A,D,E$ the set of isomorphism classes of indecomposable representations are in bijection with the set of positive roots of the corresponding root system (see \cite{gabriel}). Moreover, one can obtain an explicit description of the inner product and the root lattice. Ringel then showed that using this category one could construct the positive part of the corresponding Lie algebra (see \cite{ringel}). To obtain all roots and the whole Lie algebra one must consider isomorphism classes of objects in a related category; the ``2-periodic derived category"  $\D^{b} (\Rep (\overrightarrow{\Gamma}) )/ T^{2}$, where $T$ is the translation functor (see \cite{px}).  In this approach two different choices of orientation of the same graph give rise to different Abelian categories, which are not equivalent, but instead derived equivalent. The relation between different categories is given by the ``BGP reflection functors".

The drawback of the quiver approach is that it requires a choice of orientation of the Dynkin diagram, making the constructions non-canonical. Similarly, the standard construction of a root system requires a choice of simple roots. One would like to find a construction which does not require these choices.

Another approach, which is independent of any choice of orientation, was suggested by Ocneanu \cite{ocneanu}, in the setting of quantum subgroups of $SU(2)$. His idea was to give a purely combinatorial construction by studying ``essential paths" in the quiver $\Gammahat = \Gamma \times \Z_{h}$, which requires no choice of orientation.

In the case of affine Dynkin diagrams the McKay correspondence provides a tool for avoiding choosing orientations. The classical McKay correspondence identifies affine Dynkin diagrams of type $A,D,E$ and finite subgroups $G\subset SU(2)$. In 2006, Kirillov Jr. studied a geometric approach to McKay correspondence using $\bar{G}$-equivariant coherent sheaves on $\PP^{1}$, where $\bar{G} = G/ \pm I$ (see \cite{kirillov}).
In particular, it was shown that indecomposable objects in the category $\D^{b}_{\bar{G}}(\PP^{1})/T^{2}$ are in bijection with the roots of the corresponding affine root system, that the inner product and root lattice admit an explicit description in terms of this category, and that although there is no natural choice of simple roots, there is a canonical Coxeter element in the Weyl group. This gives a ``categorical construction" of the corresponding root system. Here ``categorical construction" means that roots are realized as classes of indecomposable objects in a certain category, and the inner product and root lattice admit explicit description in terms of this category.

Motivated by the above constructions, the authors conjectured an analogous construction to the affine case, in the case when $\Gamma$ is a finite Dynkin graph in \cite{kt}. (For the reader's convenience that conjecture is stated at the end of this section.)

In this paper, that conjecture is established. Given a Dynkin diagram $\Gamma$ we study a translation quiver $\Gammahat \subset \Gamma \times \Z$, and a triangulated subcategory $\D \subset \D (\Gammahat)$ of the corresponding derived category. Given any choice of orientation $\Om$, equivalences $\D \to \D^{b} (\Rep (\Gamma , \Om))$ are constructed and shown to be compatible with the reflection functors.

Moreover, this construction is closely related to the preprojective algebra of $\Gamma$. We begin by giving a ``graphical" description of the Koszul complex of the preprojective algebra. In this setup elements of degree k in the Koszul complex are visualized as paths in the quiver $\Gammahat$ with k ``jumps". This description is then used to construct the indecomposable objects in the category $\D$. We show that the category $\D$ has $\Gammahat ^{op}$ as its Auslander-Reiten quiver, and use this to relate $\D$ with the mesh category as described in \cite{bbk} and \cite{happel}. This leads to a periodicity result about the preprojective algebra (see Theorem~\ref{t:periodic}). Finally, we study the quotient category $\C= \D / T^{2}$, where $T$ is the translation functor, and use this to establish the conjecture stated in \cite{kt}. The Auslander-Reiten quiver of $\C$ is denoted by $\Gammahat_{cyc}$, and $\Gammahat_{cyc} \subset \Gamma \times \Z_{2h}$, where $h$ is the Coxeter number of $\Gamma$.

The main result of this paper (and statement of the conjecture in \cite{kt}) is summarized in the following Theorem.

\begin{thm}\label{t:main1}
    Given a simply-laced Dynkin diagram $\Gamma$ with Coxeter number $h$
    there exists a triangulated category $\C$ with an exact  functor
    $\C\to\C\colon \F\mapsto\F(2)$ \textup{(}``twist''\textup{)} with the
    following properties:

    \begin{enumerate}
        \item The category $\C$ is 2-periodic: $T^2=\id$
        \item For any $\F\in \C$, there is a canonical functorial isomorphism
            $\F(2h)=\F$, where $h$ is the Coxeter number of $\Gamma$.
        \item Let $\K$ be the Grothendieck group of $\C$. The corresponding root system $R$ is identified with the set $\Ind\subset \K$  of all indecomposable classes. 
        \item The map $C\colon[\F]\mapsto[\F(-2)]$ is a Coxeter
            element for  this root system.
        \item Set $\< X,Y \>_{\C} = \dim \RHom (X,Y) = \dim \Hom (X,Y) - \dim \Ext^{1} (X,Y)$ (the ``Euler form"), then the inner
            product is given by $(X,Y) = \< X,Y \>_{\C} + \< Y,X \>_{\C}$.
        \item There is a natural bijection $\Phi : \Ind \to \Gammahat_{cyc}$, between indecomposable objects and vertices in $\Gammahat_{cyc}$.
            Under this bijection, the Coxeter element $C$ defined above is
            identified with the map $\tau :(i,n)\mapsto (i,n+2)$.
            Denote the indecomposable object corresponding 
            to $q=(i,n)$ by $X_{q}$.
        \item The category $\C$ has ``Serre Duality":  $$\Hom (X,Y) = (\Ext^{1} (Y, X(-2)))^{\ast}$$
        There is also an identification
        $$\Hom(X_{q} , X_{q^{\prime}} ) =\Path(q^{\prime}, q )/ J$$
           where $\Path$ is the vector space generated by paths in $\Gammahat_{cyc}$ and $J$ is
            some explicitly described subspace. Thus the form $\< \cdot , \cdot \>$ is determined by paths in $\Gammahat_{cyc}$.     
            
           \item For every height function $\hh$ (see Section~\ref{s:gammahat} for definition), there is a derived functor $R\rho_{\hh} : \C \to
            \D^b(\Gamma , \Om_{\hh})/T^2$ which is an equivalence of
            triangulated categories.
    \end{enumerate}
\end{thm}

\begin{remark}
It has been pointed out that Kajiura, Saito and Takahashi construct a similar category using matrix factorizations (see \cite{kst}). Specifically, from each polynomial of type $A,D,E$ they construct a triangulated category $\Hall$, and equivalence $\Hall \to \D^{b} (\Gamma , \Om)$. However, our construction is quite different than their construction; we take a more combinatorial approach, using the Dynkin graph $\Gamma$ and the quiver $\Gammahat$ as a starting point, rather than the polynomial associated to $\Gamma$.
\end{remark}

%%%%%%%%%%%%%%%%%%%%%%%%%%%%%%%%%%%%%%%%%%%%%%%%%%%%%%%%%%
\section{Preliminaries - Quivers and Preprojective Algebra}\label{s:quivers}
A quiver $\overrightarrow{\Gamma}$ is an oriented graph. The vertex set is denoted by $\Gamma_{0}$ and the arrow set is denoted by $\Gamma_{1}$. In what follows a quiver is obtained by orienting a graph $\Gamma$. In such a case the quiver is denoted by $\overrightarrow{\Gamma} = (\Gamma, \Om)$ where $\Om$ is an orientation of the graph $\Gamma$. There are two functions $s,t : \Gamma_{1} \to \Gamma_{0}$, called ``source" and ``target" respectively, defined on an oriented edge $e : i\to j$ by $s(e) = i$ and $t(e) = j$.

Fix a field $\mathbb{K}$. For any quiver $\overrightarrow{\Gamma}$ let $P(\overrightarrow{\Gamma})$ be the following algebra. As an algebra it is generated by elements $\{ e \}_{e \in \Gamma_{1}} \cup \{ e_{i} \}_{i\in \Gamma_{0}}$. Here the elements $e_{i}$ are thought of as ``paths of length 0 from $i$ to $i$". Viewing a path as a sequence of edges, the multiplication of basis elements is given by concatenation of paths. 

\begin{defi} The algebra $P(\overrightarrow{\Gamma})$ defined above is called the {\em path algebra} of $\overrightarrow{\Gamma}$. It is an assiociative algebra with unit given by $1 = \sum_{i \in \Gamma_{0}} e_{i}$.
\end{defi} 

The algebra $P(\overrightarrow{\Gamma})$ is graded by path length and by the source and target of the path. This gives a decomposition 
\begin{equation}\label{e:decomp}
P (\overrightarrow{\Gamma}) = \bigoplus_{i,j \in \Gamma ; k \in \N} P_{i,j;k}
\end{equation} 
where $P_{i,j;k}$ is the space spanned by paths of length $k$ from $i$ to $j$. (Here an edge has length 1, and the idempotent corresponding to a vertex has length 0.)

The preprojective algebra of a quiver $\overrightarrow{\Gamma}$ is defined as follows:
Consider the double quiver $\overline{\Gamma}$ which has the same vertex set as $\overrightarrow{\Gamma}$ but for every arrow $e:i\to j$ there is an arrow $\overline{e} :j \to i$.
Choose a function $\epsilon : \overline{\Gamma}_{1} \to \{ \pm 1 \}$ so that $\epsilon (e) + \epsilon ( \overline{e} ) = 0$. For each vertex $i\in \overrightarrow{\Gamma}$ define $\theta_{i} \in P_{i,i;2}$ by
\begin{equation}\label{e:mesh}
\theta_{i} =\sum_{s(e)=i}  \epsilon (e)  \overline{e} e \in P_{i,i;2}
\end{equation}

\begin{defi} The {\em preprojective algebra} $\Pi (\Gamma)$ of $\Gamma$ is defined as $P(\overline{\Gamma}) / J$ where $J$ is the ideal generated by the $\theta_{i}$'s. The ideal $J$ is called the ``mesh" ideal.
\end{defi}
Note that this algebra is independent of the choice of $\epsilon$ and depends only on the underlying graph $\Gamma$, not on the orientation $\Om$. (See \cite{lusztig2} for details.)

Alternatively, we can define $\Pi (\Gamma)$ as follows:  Let $R$ be the algebra of functions $\Gamma \mapsto \mathbb{K}$ with pointwise multiplication. Let $V$ be the vector space spanned by the edges of the double quiver $\overline{\Gamma}$ and let $L$ be the $R$-submodule of $V\otimes V$ generated by the element $\theta = \sum_{i \in \Gamma} \theta_{i} $. Consider the embedding $j:L \hookrightarrow V\otimes V$. Denote by $J$ the quadratic ideal in $V\otimes V$ generated by $L$. Then the preprojective algebra of $\Gamma$ is $\Pi = T_{R} (V)/J$. (See \cite{etginz} for details.)

A representation of a quiver $\overrightarrow{\Gamma}$ is a choice of vector space $X(i)$ for every vertex in $\Gamma_{0}$ and linear map $x_{e} : X(i) \to X(j)$ for every edge $e: i \to j$. A morphism $\Phi : X \to Y$ of representations is a collection of linear maps $\Phi (i) : X(i) \to Y(i)$ such that the following diagram is commutative for every edge $e: i \to j$.
$$
\xymatrix{
X(i) \ar[r]^{x_{e}} \ar[d]^{\Phi (i)} & X(j) \ar[d]_{\Phi (j)} \\
Y(i) \ar[r]^{y_{e}} & Y(j) \\
}$$
Denote the Abelian category of representations of $\overrightarrow{\Gamma}$ by $\Rep (\overrightarrow{\Gamma})$. 

\begin{remark} Note that a representation of $\overrightarrow{\Gamma}$ is the same as a module over the path algebra $P( \overrightarrow{\Gamma} )$ and that the notion of morphism for each coincide as well. (See \cite{crawley-boevey} for details.)
\end{remark}

For each vertex $i$ define a representation $P_{i}$ by setting $P_{i} (j) = P_{i,j}$ the space spanned by paths from $i$ to $j$ in $\overrightarrow{\Gamma}$. This representation is projective and indecomposable, and any indecomposable projective is isomorphic to $P_{i}$ for some vertex $i$ (see \cite{crawley-boevey}). For any vertex $i$ define a simple object $S_{i}$ by setting $S_{i} (j) = \delta_{i,j} \mathbb{K}$.

Now consider the corresponding bounded derived category, denoted by $\D^{b} (\overrightarrow{\Gamma})$. Recall that an object in $\D^{b} (\overrightarrow{\Gamma} )$ can be thought of as a choice of bounded complex $X^{\bullet} (i)$ for each vertex $i \in \Gamma$, together with maps of complexes $x_{e} : X^{\bullet} (i) \to X^{\bullet} (j)$ for each edge $e:i \to j$.\\
In $\D^{b} (\overrightarrow{\Gamma})$ the indecomposable objects, up to isomorphism, are of the form $X[k]$, where $X$ is an indecomposable object in $\Rep (\overrightarrow{\Gamma})$ considered as a complex concentrated in degree 0. (See \cite{happel} for details.) When $\Gamma$ is Dynkin, the Auslander-Reiten quiver of the derived category is the quiver $\Gammahat$ defined in Section~\ref{s:gammahat}. For more details about the structure of the derived category see \cite{happel}.

\begin{defi}
The {\em Auslander-Reiten} quiver of a category $\A$ is defined as follows:
\begin{enumerate}
\item The vertices are the set $\Ind(\A)$ of non-zero isomorphism classes of indecomposable objects.
\item For vertices $[X] , [Y]$ there are $d_{ij}$ edges $e : [X] \to [Y]$, where $d_{ij} = \dim \text{Irr} (X, Y)$ and $\text{Irr} (X,Y)$ denotes the irreducible morphisms $X\to Y$.
\end{enumerate}
For more details, such as the definition of irreducible morphism, see \cite{ars} Chapter VII.
\end{defi}

On $\D^{b} (\overrightarrow{\Gamma})$ there are functors $\nu$ and $\tau$ defined by:
\begin{align} 
&\Hom (X,Y) = (\Ext^{1} (Y, \tau X))^{*} \\
&\nu (X) = ( \Hom_{P} (X, P) )^{*}
\end{align}
where in the second line a representation $X$ is thought of as a module over the path algebra $P = P(\overrightarrow{\Gamma})$.\\

\begin{remark}
In the setting of equivariant sheaves on $\PP^{1}$ considered in \cite{kirillov}, the functor $\tau$ is given by tensoring with the dualizing sheaf $\OO (-2)$.
\end{remark}

\subsection{BGP Reflection Functors} 
Despite the fact that Gabriel's Theorem establishes a bijection between the Grothendieck groups of $\Rep (\overrightarrow{\Gamma})$ for any choice of $\Om$, it is not the case that these categories are equivalent for different choices of $\Om$. 

Let $\overrightarrow{\Gamma} = (\Gamma , \Om)$ be a quiver and let $i \in \Gamma_{0}$ be a sink (or source) in the orientation $\Om$. Define a new orientation $s_{i} \Om$ of $\Gamma$ by reversing all arrows at $i$. Hence the sink becomes a source (and the source a sink).

Let $i \in \Gamma$ be a source (or sink) in the orientation $\Om$ and consider the categories $\Rep (\Gamma , \Om)$ and $\Rep (\Gamma , s_{i} \Om)$. These two categories are not equivalent, however there is a nice functor $S_{i}^{\pm} : \Rep (\Gamma , \Om) \to \Rep (\Gamma , s_{i} \Om)$ between them. (See \cite{bgp}.)  These functors are left (respectively right) exact. Although this functor is not an equivalence, the corresponding derived functor $RS_{i}^{+}$ (respectively $LS_{i}^{-}$) provides an equivalence of triangulated categories $\D^{b} (\Gamma, \Om) \to \D^{b} (\Gamma , s_{i} \Om)$. A brief description of the derived functors is given for the reader's convenience.
 
Let $\overrightarrow{\Gamma} = (\Gamma, \Om)$ and let $i$ be a source for $\Om$. Define a functor $RS_{i}^{+} : \D^{b} (\Gamma, \Om) \to \D^{b} (\Gamma , s_{i} \Om)$ by setting
$$ RS_{i}^{+} X (j) = \begin{cases}
Cone(X(i) \to \bigoplus_{i\to k} X(k)) &\text{if} \ \ i=j \\
X(j) &\text{otherwise.}
\end{cases}
$$ 
For an edge $e: j \to k$ in $\Om$ let $\overline{e}$ denote the corresponding edge in $s_{i} \Om$, the map $RS_{i}^{+} (x_{\overline{e}})$ is given by 
$$RS_{i}^{+} (x_{\overline{e}}) = \begin{cases}
x_{e} &\text{if } s(e) \neq i \\
(0 ,\iota_{j}) : X(j) \to X^{\bullet +1}(i) \bigoplus \oplus_{i\to j} X(j) &\text{if } s(e) = i
\end{cases}
$$
where $\iota_{j}$ is the embedding of $X(j)$ into $\oplus X(j)$.

Similarly for $i$ a sink define $LS_{i}^{-} : \D^{b} (\Gamma, \Om) \to \D^{b} (\Gamma , s_{i} \Om)$ by 
$$ LS_{i}^{-} X (j) = \begin{cases}
Cone(\bigoplus_{i\to k} X(k) \to X(i) ) &\text{if} \ \ i=j \\
X(j) &\text{otherwise.}
\end{cases}
$$ 
For an edge $e: j \to k$ in $\Om$ let $\overline{e}$ denote the corresponding edge in $s_{i} \Om$, the map $LS_{i}^{-} (x_{\overline{e}})$ is given by 
$$LS_{i}^{-} (x_{\overline{e}}) = \begin{cases}
x_{e} &\text{if } t(e) \neq i \\ 
(\iota_{j}^{+1}, 0 ) : X(j) \to ( \oplus_{j\to i} X^{\bullet +1} (j)  \bigoplus X^{\bullet}(i) ) &\text{if } t(e) = i
\end{cases}
$$
where $\iota_{j}$ is the embedding of $X(j)$ into $\oplus X(j)$.

These are the derived functors of the well-known ``BGP reflection functors" (see \cite{gelfman}). Note that the functors $RS_{i}^{+}$ and $LS_{i}^{-}$ are inverse to each other. The name reflection functor comes from the action of these functors on the Grothendieck group. In the setting of Gabriel's Theorem, the Grothendieck group of $\D^{b} (\overrightarrow{\Gamma}) /T^{2}$ is isomorphic to the root lattice, indecomposable objects correspond to the roots and the reflection functors act on the Grothendieck group as the corresponding simple reflections in the Weyl group of the associated root system.

%%%%%%%%%%%%%%%%%%%%%%%%%%%%%%%%%%%%%%%%%%%%%%%%%%%%%%%%%%%
\section{The Quiver $\Gammahat$}\label{s:gammahat}
Let $\Gamma$ be a finite graph without cycles. So in particular, $\Gamma$ is bipartite. Let $\Gamma = \Gamma_0\sqcup \Gamma_1$ be a bipartite splitting.
Define the quiver $\Gamma \times \Z$ as
follows:
\begin{align*}
&\text{vertices}\colon \Gamma\times \Z\\
&\text{edges}\colon \text{for each $n\in \Z$ and edge }i-j \text{ in
}\Gamma, \text{ there are oriented edges }\\
&\qquad(i,n)\to(j,n+1), (j,n)\to(i,n+1) \text{
in }\Z\Gamma
\end{align*}

For $\Gamma$ as above, with bipartite with splitting $\Gamma = \Gamma_0\sqcup \Gamma_1$, the quiver $\Z\Gamma$ is disconnected: $\Gamma \times \Z =(\Gamma \times \Z)_0\sqcup (\Gamma \times \Z)_1$, where 
$$(\Gamma \times \Z )_k=\{(i,n)\st n+p(i) \equiv k\mod 2\}$$ where $p(i)=0$ for $i\in \Gamma_0$ and $p(i)=1$ for $i\in \Gamma_1$.  

\begin{defi} Define the quiver $\Gammahat$ by setting 
\begin{equation}
\Gammahat = \{(i,n) \subset \Gamma \times \Z \st n+p(i) \equiv 0 \mod 2\} = (\Gamma \times \Z)_0
\end{equation}

Let $\Gamma$ be an $A,D,E$ Dynkin diagram with Coxeter number $h$, so in particular $\Gamma$ is bipartite. In this case, define a cyclic version of $\Gammahat$ by setting 
\begin{equation}\label{e:Ihat}
\Gammahat_{cyc} = \{ (i,n) \subset \Gamma \times \Z_{2h} \st n+p(i) \equiv 0 \mod 2 \}
\end{equation}
\end{defi}

\begin{example}
    For the graph $ \Gamma = D_{5}$ the quiver $\Gammahat$ is shown in
    Figure~\ref{f:Ihat-D5}. 
    \begin{figure}[ht]
    \centering
        \includegraphics[height=3.00in]{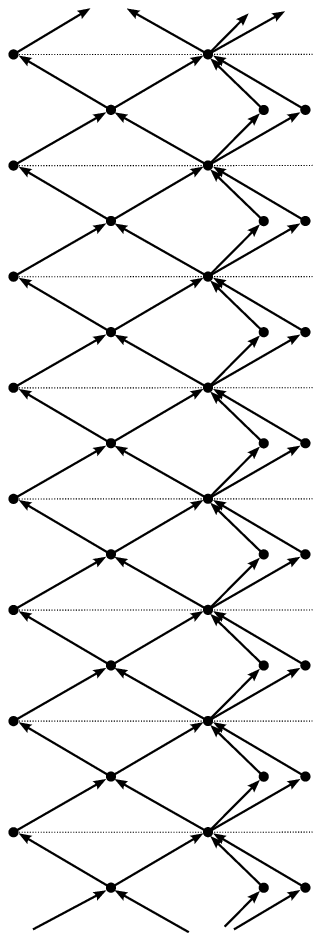}
        \caption{The quiver $\Gammahat$ for graph $\Gamma = D_{5}$. For $D_{5}$ the Coxeter number is 8, so by identifying the outgoing arrows at the top level and the incoming arrows at the bottom level in this figure, one obtains $\Gammahat_{cyc}$.}\label{f:Ihat-D5}
    \end{figure}
\end{example}

\begin{remark}
Note that the quivers $\Gammahat$ and $\Gammahat_{cyc}$ do not depend on the choice of $p$, and in fact can be canonically defined as one (of two identical) connected component of the quiver $\Z \times \Gamma $. The presentation given here is to simplify notation and make it possible to write explicit formulae.
\end{remark}

It is well-known that $\Gammahat$ and $\Gammahat_{cyc}$ are the Auslander-Reiten quivers of the categories $\D^{b}( \Gamma , \Om)$ and $\D^{b} (\Gamma , \Om) / T^{2}$ respectively, when $\Gamma$ is of type A,D,E. (See \cite{happel}.)

The following basic properties of $\Gammahat$ also hold for $\Gammahat_{cyc}$. For brevity only $\Gammahat$ is considered.

Define a ``twist" map $\tau\colon \Gammahat \to
\Gammahat$ by
\begin{equation}\label{e:tau}
\tau(i,n)=(i,n+2).
\end{equation}

\begin{defi}\label{d:height}
A function $\hh \colon \Gamma \to \Z$ satisfying $\hh (j)=\hh (i)\pm 1$ if $i,j$ are connected by an edge in $\Gamma$ and satisfying $\hh (i)\equiv p(i)\mod     2$,  will be called a {\em height function}. (Here  $p$ is the parity function defined in the beginning of this section.)
\end{defi}

\begin{defi}\label{d:slice}
Following \cite{gabriel2}, a connected full subquiver of $\Gammahat$ which contains a unique representative of $\{ (i,n) \}_{n\in \Z}$ for each $i\in \Gamma$ will be called a {\em slice}.
\end{defi}

Any height function $\hh$ defines a slice $\Gamma_{\hh} = \{(i, \hh (i))\st i\in \Gamma \}\subset \Gammahat$; it also defines an  orientation $\Om_{\hh}$ on $\Gamma$ where $i\to j$ if $i,j$ are connected by an edge and $\hh (j)= \hh (i)+1$. It is easy to see that two height functions give the same orientation if and only if they differ by an additive constant, or equivalently, if the corresponding slices are obtained one from another by applying a power of $\tau$. Conversely, the second coordinate of any slice defines a height function.

Let $\hh$ be a height function and let $i\in \Gamma$ be a source for the corresponding orientation $\Om_{\hh}$. Define a new height function $s^{+}_{i} \hh$ by 
$$s^{+}_{i} \hh (j) = \begin{cases}
\hh (j)+2 &\text{if} \ \ j=i \\
\hh (j) &\text{if} \ \ j\neq i 
\end{cases}
.$$
Similarly, if $i\in \Gamma$ is a sink for the corresponding orientation $\Om_{\hh}$, define a new height function $s^{-}_{i} \hh$ by 
$$s^{-}_{i} \hh (j) = \begin{cases}
\hh (j)-2 &\text{if} \ \ j=i \\
\hh (j) &\text{if} \ \ j\neq i 
\end{cases}
.$$

Note that the orientation $\Om_{s^{\pm}_{i} \hh}$ of $\Gamma$ is obtained by reversing all arrows at $i$, and that any orientation of $\Gamma$ can be obtained by a sequence of such operations.  It is well-known that for any two height functions $\hh, \hh^{\prime}$ one can be obtained from the other by a sequence of operations $s_{i}^{\pm}$. 

For $\Gamma$ Dynkin of type $A,D,E$, with Coxeter number $h$, define the following permutations on $\Gammahat$ (and $\Gammahat_{cyc}$). \\
The ``Nakayama" permutation given by:
\begin{equation}\label{e:nakayama}
\nu_{\Gammahat} (i,n) = (\check{\imath} \ ,n+h-2)
\end{equation}
The ``Twisted Nakayama" permuation given by:
\begin{equation}\label{e:gamma}
\gamma_{\Gammahat} (i,n) = (\check{\imath} \ , n+h) = \tau \circ \nu_{\Gammahat} (i, n)
\end{equation}
Here $\check{\imath}$ is defined by $-\alpha_{i}^{\Pi} = w_{0}^{\Pi} (\alpha_{\check{\imath}}^{\Pi})$, where $w_{0}^{\Pi} \in W$ is the longest element and $\Pi$ is any set of simple roots. Thus for the root systems of type $A, D_{2n+1}, E_{6}$ this map corresponds to the diagram automorphism, while for $D_{2n}, E_{7}, E_{8}$ this map is just the identity.

It remains to verify that $\nu_{\Gammahat}$ is well-defined. This only requires checking that the image does, in fact,
lie in $\Gammahat$.  Note that if $h$ is even $p(\check{\imath}) = p(i)$ and $k+h=k
\mod 2$, so $(\check{\imath}, k+h) \in \Gammahat$. If $h$ is odd, then $R=A_{2n}$,
$h=2n+1$ and $\check{\imath} = 2n-i+1$, so that $p(\check{\imath})=p(i) + 1$
and $k+h=k+1 \mod 2$, so  $(\check{\imath}, k+h) \in \Gammahat$. Hence the map
$\nu_{\Gammahat}$ is well-defined. 

\begin{example}
The maps $\nu_{\Gammahat}$ and $\gamma_{\Gammahat}$ for the case $\Gamma = A_{4}$ are shown in Figure~\ref{f:nugamma}.
\begin{figure}[ht]
\centering
\begin{overpic}
{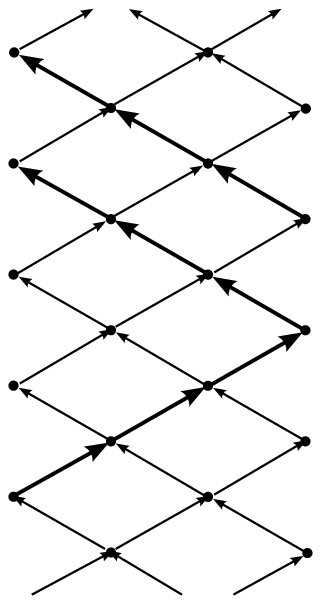}%[grid]
\put (-15,92){$\gamma_{\Gammahat} \Gamma_{\hh}$}
\put (-9,17){$\Gamma_{\hh}$}
\put (-15,72){$\nu_{\Gammahat} \Gamma_{\hh}$}
\end{overpic}

\caption{The maps $\nu_{\Gammahat}$ and $\gamma_{\Gammahat}$ in the case $\Gamma = A_{4}$. A slice $\Gamma_{\hh}$ and its images under $\nu_{\Gammahat}$ and $\gamma_{\Gammahat}$ are shown in bold.}\label{f:nugamma}
\end{figure}
\end{example}
In terms of the Auslander-Reiten quiver, these maps correspond to the functors $\tau$ and $\nu$ defined on $\Rep (\Gamma , \Om)$.

%%%%%%%%%%%%%%%%%%%%%%%%%%%%%%%%%%%%%%%%%%%%%%%%%%%%%%%%%%%%%%%%%%
\section{The Category $\D$}\label{s:catD}
In this section a precursor to the conjectured category is introduced and several basic results are established. To begin a few preliminary results are required.

\begin{defi}
Let $\overrightarrow{\Gamma}$ be any quiver and $e: i\to j$ be an edge in $\overrightarrow{\Gamma}$. For any object $X\in \D (\overrightarrow{\Gamma})$ define the complex of vector spaces $Cone_{e}(X)$ as the cone of the map $x_{e} : X(i) \to X(j)$. This will be called the ``Cone over the edge $e$".
\end{defi}

\begin{remark}
Note that $x_{e}: X(i) \to X(j)$ is an honest morphism between complexes of vector spaces, not a morphism in a derived category.
\end{remark}

\begin{lemma}\label{l:edgecone}
For any quiver $\overrightarrow{\Gamma}$ and edge $e:i \to j$, the map $Cone_{e} : \D^{b} (\overrightarrow{\Gamma}) \to \D^{b} (Vect)$ is functorial.
\end{lemma}
\begin{proof}
Suppose that $X,Y \in \D (\overrightarrow{\Gamma})$ and that $F:X \to Y$ is a map of complexes of representations. Let $x_{e} : X(i) \to X(j)$ and  $y_{e} : Y(i) \to Y(j)$ be the maps of complexes corresponding to the edge $e$. Then $F(j) \circ x_{e} = y_{e} \circ F(i)$ and $d_{Y} \circ F = F \circ d_{X}$. By definition of $Cone_{e}$, this gives a map

$$\xymatrix{
Cone_{e} (x_{e}) \ar[d]^{ Cone_{e} (F) } & = (X^{\bullet +1}(i) \oplus X(j) , d_{X}(i) + x_{e} , d_{X} (j)) \\
Cone_{e} (y_{e})  & = (Y^{\bullet +1}(i) \oplus Y(j) , d_{Y}(i) + y_{e} , d_{Y} (j)) \\
}$$

Also note that this shows that if $F$ is a quasi-isomorphism, then the induced map $Cone_{e}(F)$ is also a quasi-isomorphism. \\
Let $f \in \Hom_{\D (\overrightarrow{\Gamma})} (X, Y)$ and let $X \leftarrow Z \to Y$ be a roof diagram for $f$. Then the above shows that there is an induced roof diagram $Cone_{e} (x_{e}) \leftarrow Cone_{e} (z_{e}) \to Cone_{e} (y_{e})$. Hence $Cone_{e}$ is functorial.
\end{proof}

Let $\Gamma$ be a finite graph without cycles. Let $\D (\Gammahat)$ be the (unbounded) derived category of representations of $\Gammahat$.

\begin{defi}
Define a category $\D$ defined as follows: $Obj_{\D} = \{ (X , \Phi ) \}$ where $X \in \D (\Gammahat )$ is an object such that for any $q \in \Gammahat$ the complex $X(q)$ is bounded, and $\Phi$ is a collection of isomorphisms of complexes of vector spaces $\phi_{q} : X (\tau q) \to Cone(x_{q})$ for each $q \in \Gammahat$ such that the following diagram commutes. 

$$\xymatrix{
\bigoplus_{q^{\prime}: q\to q^{\prime}} X(q^{\prime}) \ar[d]_{\sum_{q^{\prime} : q\to q^{\prime}} x_{q^{\prime}}} \ar[dr]^{\iota} & \\
X(\tau q) \ar[r]_{\phi_{q}} & Cone (x_{q}) \\
}$$

Here $x_{q} : X(q) \to \bigoplus_{q\to q^{\prime}} X(q^{\prime})$ is the map given by edges $q\to q^{\prime}$, $x_{q^{\prime}} : X(q^{\prime}) \to X(\tau q)$ is the map given by the edge $q^{\prime} \to \tau q$, $\iota$ is the inclusion map and $Cone(x_{q}) = \bigoplus_{e: s(e)=q} Cone_{e} (x_{e})$ is the direct sum of Cone over the edges $e: q\to q^{\prime}$, as defined above.

A morphism $F : (X, \Phi) \to (Y, \Psi )$ in $\D$ is given by a morphism $F : X \to Y$ in $\D (\Gammahat)$ such that the following diagram is commutative: 
\begin{equation}\label{e:morphism}
\xymatrix{
X(\tau q) \ar[r]^{F(\tau q)} \ar[d]^{\phi_{q}} & Y(\tau q) \ar[d]^{\psi_{q}} \\
Cone (x_{q} ) \ \ \ \ar[r]^{Cone( F( \tau q) ) } & \ \ \ Cone(y_{q}) \\
}
\end{equation}

\begin{defi} An object $X \in \D (\Gammahat)$ is said to satisfy the {\em fundamental relation} if for any $q \in \Gammahat$ there is a choice of map of complexes $z : X(\tau q) \to X(q) [1]$ so that 
\begin{equation}\label{e:fundamentalrelation}
X(q) \stackrel{x_{q}}{\to} \bigoplus_{q \to q^{\prime}} X(q^{\prime}) \stackrel{\sum x_{q^{\prime}}}{\to} X(\tau q) \stackrel{z}{\to} X(q)[1]
\end{equation}
is an exact triangle. Here $x_{q}, x_{q^{\prime}}$ are the maps corresponding to edges $q\to q^{\prime}$ and $q^{\prime} \to \tau q$ in $\Gammahat$.
\end{defi}

\begin{prop}
For any object $(X, \Phi) \in \D$, $X$ satisfies the fundamental relation.
\end{prop}

\begin{proof}
By definition of an object $(X , \Phi) \in \D$.
\end{proof}

It remains to show that the category $\D$ inherits the structure of a triangulated category from $\D (\Gammahat)$. To do this define the translation functor in $\D$ as follows: $T_{\D} (X, \Phi) = (T_{\D (\Gammahat)} X, \Phi [1])$, where $\Phi [1] = \{ \phi_{q} [1] \}$. The distinguished triangles in $\D$ are triples $( (X, \Phi), (Y, \Psi) , (Z, \varphi) )$ where $(X,Y,Z)$ is a  distinguished triangle in $\D (\Gammahat)$.
\end{defi}

\begin{lemma}\label{l:cone}
\par\indent
\begin{enumerate}
\item If $(X, \Phi), (Y, \Psi) \in \D$ and $F:X \to Y$ is a map of complexes satisfying Equation~\ref{e:morphism}, then $( Cone(F) , \varphi) \in \D$, where the maps $\varphi_{q}$ are those induced by $\phi_{q}, \psi_{q}$.
\item Any distiguished triangle in $\D$ is isomorphic to one of the form 
$$ (X, \Phi) \stackrel{F}{\to} (Y, \Psi) \to (Cone(F), \varphi) \to X[1],$$
where $F$ is a map of complexes.
\end{enumerate}
\end{lemma}

\begin{proof}
First note that any distinguished triangle in $\D (\Gammahat)$ is isomorphic to a triangle of the form $X \stackrel{F}{\to} Y \to Cone(F) \to X[1]$, where $F: X \to Y$ is a morphism of complexes (see \cite{gelfman} Chapter IV \S 2). So once we show the first statement, the second follows.

It remains to verify that if $F : X \to Y$ is a morphism of complexes between objects in $\D$ then $Z=Cone(F)$ comes with an identification $Z(\tau q) \to Cone(z_{q})$ where $z_{q}$ is the map $Z(q) \to \oplus_{q\to q^{\prime}} Z (q^{\prime})$. To see this note that:

\begin{align*}
&Cone^{k} (z_{q}) = \\
&= Cone^{k} (Z(q) \to \oplus_{q \to q^{\prime}} Z(q^{\prime})) \\
&= Z^{k+1}(q) \bigoplus  (\oplus_{q\to q^{\prime}} Z^{k}(q^{\prime}) )\\
&= Cone^{k+1} (F(q)) \bigoplus (\oplus_{q\to q^{\prime}} Cone^{k} (F(q^{\prime})))\\
&= \big\{ X^{k+2} (q) \oplus Y^{k+1} (q) ) \big\} \bigoplus \big\{ (\oplus_{q\to q^{\prime}} ( X^{k+1} (q^{\prime}) \oplus Y^{k} (q^{\prime}) )\big\} \\
&= \big\{ X^{k+2} (q) \oplus ( \oplus_{q\to q^{\prime}} X^{k+1} (q^{\prime})) \big\}  \bigoplus  \big\{ Y^{k+1} (q) \oplus ( \oplus_{q\to q^{\prime}} Y^{k} (q^{\prime}))\big\} \\
&= Cone^{k+1} ( X(q) \to \oplus_{q\to q^{\prime}} X(q^{\prime})) \bigoplus Cone^{k} ( Y(q) \to \oplus_{q\to q^{\prime}} Y(q^{\prime})) \\
&\simeq X^{k+1} (\tau q) \bigoplus Y^{k} (\tau q)  \text{ \ \ (using the isomorphisms } \phi_{q}, \psi_{q}) \\
&=Cone^{k} (F(\tau q))\\
&=Z^{k}(\tau q).
\end{align*}

Denote this isomorphism by $\varphi_{q}$. To check that this is an isomorphism of complexes, not just of graded vector spaces, it remains to check that the differentials match. 
Let $\delta$ be the differential of $Cone(F)$ and let $d_{X}, d_{Y}$ be the differentials of $X,Y$. Denote by $D$ the differential of $Cone (X^{k +1} (q) \oplus Y^{k} (q) \stackrel{x_{q} + y_{q}}{\to} \bigoplus_{q\to q^{\prime}} (X^{k +1} (q^{\prime}) \oplus Y^{k} (q^{\prime}) ) )$.

\begin{align*}
\delta (\tau q) &= ( d_{X}^{+1} (\tau q) + F^{+1} (\tau q) , d_{Y} (\tau q)) \\
&= ( d_{X}^{+2} (q) + x_{q}^{+1} + F^{+1} (q) , \sum_{q \to q^{\prime}} d_{X}^{+1} + F^{+1} (q^{\prime}) , d_{Y}^{+1} (q) + y_{q} , \sum_{q \to q^{\prime}} d_{Y} (q^{\prime})) \\
&= ( d_{X}^{+2} (q) + x_{q}^{+1} + F^{+1} (q) , d_{Y}^{+1} (q) + y_{q} , \sum_{q \to q^{\prime}} d_{X}^{+1} + F^{+1} (q^{\prime}) , \sum_{q \to q^{\prime}} d_{Y} (q^{\prime})) \\
&= ( \delta^{+1} (q) + x_{q}^{+1} + y_{q}^{+1} , \delta (q^{\prime})) \\
&= D
\end{align*}

Hence $(Cone(F), \varphi) \in \D$.

\end{proof}

\begin{thm}\label{t:triangulated}
$\D$ is a triangulated category.
\end{thm}
Following \cite{gelfman}, the notation T1, T2, T3, T4 for the axioms of a triangulated category will be used for simplicity. The reader can refer to Chapter IV of \cite{gelfman} for details.

\begin{proof}
Since triangles and morphisms have been defined above, only the axioms T1 $\to$ T4 remain to be verified. 

For T1 the only thing that needs to be checked is that a morphism $(X, \Phi) \to (Y, \Psi)$ can be completed to a triangle. First note that by construction, the objects and morphisms in $\D$ are objects and morphisms in the derived category $\D (\Gammahat)$ with extra structure.  So to show that a morphism can be completed in $\D$ it is enough to show that the completion in $\D (\Gammahat)$ carries the required extra structure. By construction of the derived category any morphism can be completed to a distinguished triangle which is isomorphic to a triangle of the form
$$ X \stackrel{F}{\to} Y \to Cone(F) \to X[1]$$
where $F: X \to Y$ is a morphism of complexes (see \cite{gelfman} Chapter 4 \S 2). If $(X, \Phi) ,(Y, \Psi) \in \D$ and $F$ sastifies Equation~\ref{e:morphism}, then Lemma~\ref{l:cone} implies that this completion lies in $\D$.

T2 follows by definition of triangles in $\D$.

For T3, consider triangles $X \to Y \to Z \to X[1]$, $X^{\prime} \to Y^{\prime} \to Z^{\prime} \to X^{\prime} [1]$ in $\D$,  where notation has been simplified by replacing $(X, \Phi)$ with $X$, etc. Given $F,G$ in the diagram below, we need to verify that there exists a morphism $H$ in $\D$ making the diagram commute.

$$\xymatrix{
X \ar[r]^{u} \ar[d]_{F} & Y \ar[r]^{v} \ar[d]_{G} & Z \ar[r]^{w} \ar[d]_{H} & X[1] \ar[d]_{F[1]} \\
X^{\prime} \ar[r]^{u^{\prime}} & Y^{\prime} \ar[r]^{v^{\prime}} & Z^{\prime} \ar[r]^{w^{\prime}} & X^{\prime}\\
}$$

Take roof diagrams $X \stackrel{r}{\leftarrow} X^{\prime \prime} \stackrel{f}{\to} X^{\prime}$, $Y \stackrel{s}{\leftarrow} Y^{\prime \prime} \stackrel{g}{\to} Y^{\prime}$ representing $F,G$ respectively. Then there is a map $u^{\prime \prime} : X^{\prime \prime} \to Y^{\prime \prime}$, which can be completed to a triangle $X^{\prime \prime} \to Y^{\prime \prime} \to Z^{\prime \prime} \to X^{\prime \prime}[1]$.

Using Lemma~\ref{l:cone}, take $Z,Z^{\prime}, Z^{\prime \prime}$ to be $Cone (u), Cone (u^{\prime}), Cone (u^{\prime \prime})$ and $u, u^{\prime}, u^{\prime \prime}$ to be maps of complexes. Then we have the following commutative diagram of complexes:

$$\xymatrix{
X \ar[r]^{u} & Y \ar[r] & Cone (u) \\
X^{\prime \prime} \ar[r]^{u^{\prime \prime}} \ar[d]_{f} \ar[u]^{r} & Y^{\prime \prime} \ar[r] \ar[d]_{g} \ar[u]^{s} & Cone (u^{\prime \prime}) \ar[d]_{f[1] \oplus g}  \ar[u]^{r[1] \oplus s} \\
X^{\prime} \ar[r]^{u^{\prime}} & Y^{\prime} \ar[r] & Cone (u^{\prime} )\\
}$$

Then the required morphism $Cone (u) \to Cone (u^{\prime})$ is represented by the third column of the commutative diagram, and since $f,g,r,s$ each satisfy Equation~\ref{e:morphism}, so do these maps. Hence the completion is in $\D$.

For T4 (Octahedron Axiom), again Lemma~\ref{l:cone} shows that by completing the diagram in $\D (\Gammahat)$, the completion lies in $\D$. To see this, suppose we are given the upper cap of an octahedron in $\D$, as pictured below. Here all objects and maps are in $\D$, and the top and bottom triangles are distinguished in $\D$.
$$\xymatrix{
X^{\prime} \ar[dd]_{[1]} \ar[dr]^{[1]} & & Z \ar[ll] \\
 & Y \ar[dl] \ar[ur]^{v} & \\
 Z^{\prime} \ar[rr]_{[1]} & & X \ar[ul]^{u} \ar[uu]_{u \circ v} \\
}$$

 Note that by completing the lower cap of the octahedron in $\D (\Gammahat)$, we get the following diagram, where the outer morphisms are in $\D$, the left and right triangles are distinguished (though not necessarily in $\D$), and the other two triangles commute.

$$\xymatrix{
X^{\prime} \ar[dd]_{[1]} & & Z \ar[ll] \ar[dl] \\
 & Y^{\prime} \ar[ul] \ar[dr]^{[1]} & \\
 Z^{\prime} \ar[ur] \ar[rr]_{[1]} & & X \ar[uu]_{u \circ v} \\
}$$
Since $X \stackrel{u\circ v}{\to} Z \to Y^{\prime}$ is distiguished, $Y^{\prime}$ is isomorphic to $Cone(u\circ v)$, and since $u\circ v$ is a morphism in $\D$, Lemma~\ref{l:cone} implies $Y^{\prime} \in \D$, and that the morphisms in the right triangle are in $\D$. Similarly, the left triangle is also in $\D$. Hence the lower cap can be completed to an octahedron, where all objects and morphisms lie in $\D$.
\end{proof}

From now on, we will abuse notation, and denote an object $(X, \Phi) \in \D$ simply by $X$ to simplify notation whenever possible.

%%%%%%%%%%%%%%%%%%%%%%%%%%%%%%%%%%%%%%%%%%%%%%%%%%%%%%%%%%
\section{dg-Preprojective Algebra}\label{s:preproj}
In this section a graphical description of the ``derived preprojective algebra" is given. This is then related to the Koszul complex of the preprojective algebra. Later this will be used to construct projectives in the category $\Rep (\Gammahat)$ and to define indecomposable objects in the category $\D$. This algebra is known to experts, however the presentation given here is not readily available in the literature. 

To begin consider the following algebra $A$.
\begin{defi}\label{d:A}
Let $P$ be the path algebra of $\overline{\Gamma}$. Let $A$ be the algebra obtained by adjoining to $P$ generators $\{ l_i \}_{i\in \Gamma},$ with relations 
$$
l_ie_j=e_j l_i=\delta_{ij} l_i.
$$

\end{defi}
Thus, $A$ is generated by the expressions of the form 
\begin{equation}\label{e:gen_of_A}
p_1l_{i_1}p_2\dots p_k l_{i_k}p_{k+1}
\end{equation}
where each $p_a$ is a path from $i_a$ to $i_{a-1}$. 

The elements of $A$ are pictured as paths in $\Gammahat$ with jumps $(i,n)$ to $(i,n+2)$ for each $l_{i}$ that appears.

Extend the grading of $P$ (see Equation~\ref{e:decomp}) to $A$ by letting the $l_i$  be elements of degree 2, so that $l_i\in A_{i,i;2}$. This gives a decomposition of $A$:
\begin{equation}\label{e:A}
A=\bigoplus_{i,j\in \Gamma \ , \ k\in \Z_+} A_{i,j;k}
\end{equation}

The interesting fact is that $A$ has another grading, which is given by the number of jumps (i.e. the number of $l_i$ which appear in an expression). This gives a further decomposition
of $A$:
\begin{equation}\label{e:Agrading}
A=\bigoplus_{n\le 0}A^n=\bigoplus_{i,j\in \Gamma , \ k\in \Z_+, \ n\le 0}
A^n_{i,j;k}
\end{equation}
where $A^{-n}$ is the subspace in $A$ generated by expressions of the
form \eqref{e:gen_of_A} with exactly $n$ $l_i$'s. In particular, $A_{i,j;l}^{-n}$ can be thought of as paths from $i$ to $j$, of length $k$, with $n$ jumps.

The graded vector space $A = \bigoplus A^{n}$ is made into a complex by setting
\begin{equation}\label{e:d}
\begin{aligned}
d_{-k}\colon A^{-k}&\to A^{-k+1}\\
p_1l_{i_1}p_2\dots p_k l_{i_k}p_{k+1}&\mapsto 
\sum_{a=1}^{k} (-1)^{a+1} p_1l_{i_1}p_2\dots p_a\th_{i_a}p_{a+1}\dots  
p_k l_{i_k}p_{k+1}
\end{aligned}
\end{equation}
where $\theta_{i} \in A_{i,i;2}$ is given by 
\begin{equation}\label{e:mesh}
\theta_{i} =\sum_{j} \epsilon (e^{ij}) e^{ji} e^{ij} \in P_{i,i;2}
\end{equation}
where the sum is over all $j$ connected to $i$ in $\Gamma$, and $e^{ij}$ denotes the oriented edge from $i$ to $j$ in $\overline{\Gamma}$.
A routine calculation shows that this definition of $d$ and multiplication in $A$ make $A$ a dg-algebra.

This dg-algebra can be easily identified with the Koszul complex of the preprojective algebra $\Pi (\Gamma)$. Recall the description $\Pi (\Gamma) = T_{R} (V)/J$ given in Section~\ref{s:quivers}. The Koszul complex of $\Pi (\Gamma)$ is given by $$K_{\bullet} = T_{R}(V\oplus L) = \bigoplus V^{n_{1}} \otimes L \otimes V^{n_{2}} \otimes L \otimes \cdots L \otimes V^{n_{j}}$$ where $V^{n} = V^{\otimes n}$. (For details see \cite{etginz}.) The differential $d$ is given by 
\begin{equation}\label{e:Kdiff}
d(v_{1}\otimes l_{1} \otimes v_{2} \otimes \cdots \otimes l_{j-1} \otimes v_{j} ) = \sum_{i} (-1)^{i} v_{1}\otimes l_{1} \otimes v_{2} \otimes \cdots v_{i} \otimes j(l_{i}) \otimes v_{i+1} \otimes \cdots  \otimes l_{j-1} \otimes v_{j}
\end{equation}
where $v_{1}\otimes l_{1} \otimes v_{2} \otimes \cdots \otimes l_{j-1} \otimes v_{j} \in V^{n_{1}} \otimes L \otimes V^{n_{2}} \otimes L \otimes \cdots L \otimes V^{n_{j}}$.

To relate these two constructions, we simply identify an element $a_{k} \otimes a_{k-1} \otimes \cdots \otimes a_{1} \in V^{k}$ with the path $p=a_{k}a_{k-1} \cdots a_{1}$, and the element $j(l_{i})$ with $\theta_{i}$. It is easily verified that this identification is compatible with differentials and multiplication, which gives an identification $K_{\bullet} \simeq A^{\bullet}$ as dg-algebras. In particular, this identification gives a combinatorial picture of the Koszul complex of the preprojective algebra.

\begin{subsection}{The non-Dynkin case}
Consider the case where $\Gamma$ is non-Dynkin. 
It is known (see \cite{mv} ) that the preprojective algebra of a non-Dynkin graph is Koszul, and that the Koszul complex $K_{\bullet}$ gives a dg-algebra resolution of $\Pi$. These results, together with the identification $ K_{\bullet} \simeq A^{\bullet}$ gives the following result.

\begin{prop}
Suppose $\Gamma$ is non-Dynkin. Then 
$$H^{k} (A^{\bullet}) =\begin{cases} 
\Pi &\text{if} \ \ k=0 \\
0 &\text{if} \ \ k > 0 
\end{cases}$$
so the complex $A^{\bullet}$ gives a dg-algebra resolution of the preprojective algebra $\Pi$.
\end{prop}
\end{subsection}

%%%%%%%%%%%%%%%%%%%%%%%%%%%%%%%%%%%%%%%%%%%%%%%%%%%%%%%%%%%
\section{Projective Representations of $\Gammahat$}\label{s:projreps}
Let $q\in \Gammahat$ be any vertex. Let $\Path^{k}(q,v)$ denote the space of ``paths with k jumps" from q to v. For $q=(i,n)$ and $v=(j,m)$, this space can be identified with the component $A_{i, j ; m-n}^{-k}$, defined by  Equation~\ref{e:Agrading}. 
\\
Define a representation $X_{q}^{k} \in \Rep (\Gammahat)$ by setting $X_{q}^{k} (v) = \Path^{k}(q,v)$. Composition of paths makes it a module over the path algebra $P$.
\begin{prop}\label{p:XHom}
\par\indent
\begin{enumerate}
\item For any $k$, and any vertex $q\in \Gammahat$, the representation $X_{q}^{k}$ is projective.
\item For any object $X\in \Rep (\Gammahat)$ we have $$\Hom_{\Gammahat} (X_{q}^{0}, X) = X(q).$$ 
\item For any object $X\in \Rep (\Gammahat)$ we have $$\Hom_{\Gammahat} (X_{q}^{k}, X) = \bigoplus_{v\in \Gammahat}  \Hom_{\mathbb{C}} ( X^{k-1} (q,v) , X(\tau v)).$$
\end{enumerate}
\end{prop}

\begin{proof}
\par\indent
\begin{enumerate}
\item The space $X_{q}^{k}$ is freely generated over the path algebra $P$ by elements of the form $p=t_{k}p_{k-1} \cdots p_{2}t_{1}p_{1} $ where the $t_{i}$'s are jumps and the $p_{i}$'s are paths.
\item For any $x\in X$ define $\phi_{x} : X_{q}^{0} \to X$ by $\phi_{x} (p) = p.x$. This gives the required isomorphism.
\item First the isomorphism $\Hom_{\Gammahat} (X^{k}_{q} , X) \simeq \bigoplus_{v\in \Gammahat} X(2) \otimes (X^{k-1} (q,v))^{*} $ is established. This isomorphism is given by $$\phi \mapsto \bigoplus_{p\in X^{k-1}(q, v )} \phi (tp) \otimes p^{*}$$ with inverse $$x\otimes \psi_{p} \mapsto (p_{1} t p_{2} \stackrel {\phi_{x, \psi_{p}}} {\mapsto} p_{1} x \psi_{p} (p_{2})).$$ To see this, note that any element $\phi \in \Hom_{\Gammahat} (X_{q}^{k} , X)$ is determined by where it sends the generators $t_{i_{k}}p_{k-1} \cdots p_{2}t_{i_{1}}p_{1}$. So for each path $p:q\to v$ with $k-1$ jumps we need to assign an element $x\in \tau v$ which is the value of $\phi (tp)$.
\\
To establish the desired isomorphism, use the standard identification $W \otimes V^{*} \simeq \Hom(V,W)$ to obtain  $$\Hom_{\Gammahat} (X_{q}^{k}, X) = \bigoplus_{v\in \Gammahat}  \Hom_{\mathbb{C}} ( X^{k-1} (q,v) , X(\tau v)).$$
\end{enumerate}
\end{proof}

%%%%%%%%%%%%%%%%%%%%%%%%%%%%%%%%%%%%%%%%%%%%%%%%%%%%%%%%%%%
\section{Indecomposable Objects in $\D$}\label{s:indobj}
In this section the graded components of the dg-algebra $A$ defined in Section~\ref{s:preproj} are used to define objects in $\D$. To do this, the following preliminary result is required.

\begin{lemma}\label{l:determined}
Let $\hh$ be a height function, and let $\Gamma_{\hh}$ be the corresponding slice.
\begin{enumerate}
\item
An object $X \in \D$ is determined up to isomorphism by the collection $\{ X^{\bullet} (q) \}_{q\in \Gamma_{\hh}}$ and morphisms corresponding to edges in the slice $\Gamma_{\hh}$.
\item For any $X,Y \in \D$ a morphism $f\in \Hom_{\D} (X,Y)$ is determined by the collection $\{ f(q) \}_{q\in \Gamma_{\hh}}$.
\end{enumerate}
\end{lemma}

\begin{proof} 
\par\indent
\begin{enumerate}
\item
Let $q\in \Gamma_{\hh}$ be a source. Recall that by definition of objects in $\D$, $X$ satisfies the fundamental relations and comes with a fixed isomorphism $X(\tau q) \stackrel{\sim}{\to} Cone (x_{q})$, where $x_{q}$ is the map $X(q) \to \bigoplus_{q \to q^{\prime}} X( q^{\prime} )$. Hence the complex $X^{\bullet} (\tau q)$ is determined by $X(q)$ and $X(p)$ for $q \to p$ in $\Gammahat$ and morphisms corresponding to the edges joining them. (Noting that the $X(p)$ are in $\{ X^{\bullet} (q) \}_{q\in \Gamma_{\hh}}$ since $q$ is a source.) Write $q=(i,n)$ so that $i$ is a source in the quiver $(\Gamma , \Om_{\hh})$ determined by the height function $\hh$. Apply the reflection $s_{i}$ and consider a source in $q^{\prime} \in \Gamma_{s_{i} \hh}$. Then repeating the argument above and noting that $X(q^{\prime})$ is in the collection 
 $\{ X^{\bullet} (q) \}_{q\in \Gamma_{\hh}}$, and that $X^{\bullet} (p)$ for $q^{\prime} \to p$ is in the collection  $\{ X^{\bullet} (q) \}_{q\in \Gamma_{h}} \bigcup X^{\bullet} (\tau q)$, one sees that $X^{\bullet} (\tau q^{\prime} )$ is determined. Continuing in this way it follows that for any $p = \tau^{k} q$ for $q \in \Gamma_{\hh}$ the complex $X^{\bullet} (p)$ is determined by the collection  $\{ X^{\bullet} (q) \}_{q\in \Gamma_{\hh}}$.
 
 A similar argument for $q \in \Gamma_{\hh}$ a sink can be repeated. This shows that for any $p = \tau^{-k} (q)$ with $q\in \Gamma_{\hh}$  the complex $X^{\bullet} (p)$ is determined by the collection  $\{ X^{\bullet} (q) \}_{q\in \Gamma_{\hh}}$.

\item Suppose that $F(q) : X(q) \to Y(q)$ is given for all $q \in \Gamma_{\hh}$. Take $q\in \Gamma_{\hh}$ to be a source, so that for any edge $q \to q^{\prime}$ in $\Gammahat$ , $q^{\prime}$ belongs to the slice $\Gamma_{\hh}$. Then using the isomorphisms $X( \tau q) \stackrel{\sim}{\to} Cone (x_{q})$ and $Y( \tau q) \stackrel{\sim}{\to} Cone(y_{q})$, together with the functoriality of ``cone over an edge", the following diagram has a unique completion $Cone (F (\tau q))$ making it commutative, which extends $F$ to $\tau q$.
$$
\xymatrix{
X (\tau q) \ar[d] \ar[r]^{F (\tau q)} & Y ( \tau q) \ar[d] \\
Cone (x_{q} ) \ \ \ \ar[r]^{Cone( F( \tau q) ) } & \ \ \ Cone(y_{q}) \\
}$$
Continuing in this way (and using a similiar argument for $q$ a sink) it is possible to extend $F$ to all vertices in $\Gammahat$. 
\end{enumerate} 
\end{proof}

Using the components of the dg-algebra $A$ defined in Section~\ref{s:preproj} define, for each vertex $q\in \Gammahat$, an object $X_{q}^{\bullet} \in \D$ as follows:
For $q=(i,n)$ and $v=(j,m)$ and $n \leq m$ set
 $$X_{q}^{k}(v) = A_{i, j ; m-n}^{-k}$$ where $A_{i, j ; m-n}^{-k}$ is defined by  ~\ref{e:Agrading}.  For $n > m$ notice that Lemma~\ref{l:determined} implies that this is adequate to extend to all other vertices. It remains to check this does, in fact, define an object in $\D$.

\begin{prop}\label{p:XinD}
For $q \in \Gammahat$ there is a canonical isomorphism (up to choice of function $\epsilon$) $X_{q}(\tau v) \simeq \text{Cone} (X_{q}(v) \to \oplus_{v \to v^{\prime}} X_{q} (v^{\prime}))$ and hence $X_{q} \in \D$.
\end{prop}

\begin{proof}
Let $v=(i,n) \in \Gammahat$. For any edge $e: v\to v^{\prime}$ in $\Gammahat$, denote by $\overline{e}$ the corresponding edge $v^{\prime} \to \tau v$. Define the map $\phi _{v} : Cone (x_{v}) \to X (\tau v)$ by
\begin{equation}
\phi_{v} (x,y) = t_{i}x + \sum_{e: s(e)=q} \epsilon(e) \overline{e} y
\end{equation}
where $x \in X_{q}^{\bullet +1} (v)$ and $y \in \bigoplus_{e: v \to v^{\prime}} X_{q} (v^{\prime})$. Note that the choice of sign $\epsilon (e)$ is forced by requiring that this map agree with the differentials:
\begin{align*}
\phi_{v} (d_{C} (x,y)) &= \phi_{v} (d_{X_{q}} x , (-1)^{k} x_{v}x + d_{X_{q}} y) \\
&= t_{i}d_{X_{q}} (x) + (-1)^{k} \sum_{e : s(e)=v}  \epsilon (e) \overline{e} x_{v} (x) + \epsilon (e) \overline{e} d_{X_{q}}(y) \\
&= t_{i}d_{X_{q}} (x) + (-1)^{k} \sum_{e : s(e)=v}  \epsilon (e) \overline{e} ex + \epsilon (e) \overline{e} d_{X_{q}}(y) \\
&= t_{i} d_{X_{q}} (x) + (-1)^{k} \theta_{i}x +  \epsilon (e) \overline{e} y \\
&= d_{X_{q}} (t_{i} x + \epsilon (e) \overline{e} y ) \\
&= d_{X_{q}} (\phi_{v} (x,y) )
\end{align*}
where $(x,y) \in Cone^{k} (x_{v})$ and $d_{C}$ denotes the differential on $Cone (x_{v})$.

Since paths with jumps form a basis and any path $q \to \tau v$ with $k$ jumps is either a path $p:q \to v^{\prime}$ with k jumps followed by the edge $\overline{e} : v^{\prime} \to \tau v$, or is a path $p:q\to v$ with $k-1$ jumps followed by the jump $t_{i}$, the above map gives an isomorphism of complexes.
\end{proof}

\begin{remark}
Although the object $X_{q}$ requires the choice of signs $\epsilon (e)$, different choices result in isomorphic objects. In particular, the category $\D$ does not depend on such a choice. The choice of $\epsilon$ amounts to choosing a representative of the isomorphism class of indecomposable object $[X_{q}]$ corresponding to $q\in \Gammahat$.
\end{remark}

Alternatively, for $i\neq j$ let $q=(i,n)$, then for $p=(j,n)$ set $X_{q} (p) = 0$. Now let $p=(j,n+1)$ and $n_{ij} =1$ so that $q\to p$ in $\Gammahat$. We define $X_{q}^{\bullet} (p) := \Path_{\Gammahat} (q,p)$ where by this we mean a complex with $\Path$ in degree 0, and $0$ in all other degrees. Note that by Lemma~\ref{l:determined} this is sufficient to extend to all other vertices using the fundamental relation. Note that it is clear from this definition that $X_{q}$ is indecomposable.

%%%%%%%%%%%%%%%%%%%%%%%%%%%%%%%%%%%%%%%%%%%%%%%%%%%%%%%%%%%
\section{Some results about $\Hom$ in $\D$}\label{s:homD}
In this section we give some results which will be useful in future sections.

\begin{thm}\label{t:RHom}
\par\indent
\begin{enumerate}
\item Let $Y\in \D$, and let $q \in \Gammahat$. Then there is an isomorphism $\RHom (X_{q}, Y) = Y(q)$
\item Let $q=(j,n)$, $q^{\prime} = (i,m)$, then $\RHom (X_{q} , X_{q^{\prime}}) = A_{i,j;n-m}$.
\item $\Hom (X_{q} , X_{q^{\prime}}) = \Path_{\Gammahat} (q^{\prime} , q) / J$ where $J$ is the mesh ideal, generated by the mesh relations (see Equation~\ref{e:mesh}).
\item Let $\hh$ be a height function, and $\Gamma_{\hh}$ the corresponding slice. If $q,q^{\prime} \in \Gamma_{\hh}$ then $\Ext^{i} (X_{q} , X_{q^{\prime}}) = 0$ for $i>0$.
\end{enumerate}
\end{thm}

\begin{proof}
\par\indent
\begin{enumerate}
\item Let $\Gamma_{\hh}$ be a slice through $q$. By Lemma~\ref{l:determined} Part 2, $\RHom_{\D} (X_{q} , Y)$ is determined on the slice $\Gamma_{\hh}$. On the slice $\Gamma_{\hh}$ the object $X_{q}$ is concentrated in degree 0 so we can identify $\RHom (X_{q} , Y)$ with $Y(q)$ by definition of $\RHom$ and Proposition~\ref{p:XHom} Part 2.
\item By Part 1 we have $\RHom (X_{q} , X_{q^{\prime}}) = X_{q^{\prime}} (q) = A_{i,j;n-m}$.
\item By Part 1 we have $$\Hom (X_q , X_{q^{\prime}}) = H^{0} (X_{q^{\prime}} (q)) = \Path (q^{\prime} ,q) /J.$$
\item By Part 1 we have that $\Ext^{k} (X_q , X_{q^{\prime}} ) = \Path^{k} (q^{\prime} , q) /J$. However if $q,q^{\prime} \in \Gamma_{\hh}$ then there are no paths with jumps $q^{\prime} \to q$, in other words the complex $X_{q^{\prime}} (q)$ is concentrated in degree 0.
\end{enumerate}
\end{proof}

%%%%%%%%%%%%%%%%%%%%%%%%%%%%%%%%%%%%%%%%%%%%%%%%%%%%%%%%%%%
\section{Equivalence of Categories}\label{s:equiv}
In this section, for every height function $\hh : \Gamma \to \Z$ an equivalence of triangulated categories $R \rho_{\hh} :\D \to \D^{b} (\Gamma, \Om_{\hh})$ is constructed and shown to be compatible with the reflection functors.\\

\begin{remark}
Note that the equivalence of categories given by a height function here, is between the category $\D$ and $\D^{b} ( \Gamma , \Om_{\hh})$, whereas in the case of equivariant sheaves on $\PP^{1}$, considered in \cite{kirillov}, the equivalence is between $\D$ and $\D^{b} ( \Gamma , \Om_{\hh}^{op})$ and is given by constructing a tilting object. That can also be done here, however that is not the approach taken.
\end{remark}

Recall that any height function $\hh$ determines an orientation $\Om_{\hh}$, and that the corresponding slice $\Gamma_{\hh}$ is an embedding of the quiver $(\Gamma , \Om_{\hh})$ in $\Gammahat$. So any representation of $\Gammahat$ gives a representation of $(\Gamma , \Om_{\hh})$ by restriction to the slice. So there is a restriction functor $\rho_{\hh} : \Rep(\Gammahat) \to \Rep (\Gamma, \Om_{\hh})$ defined by 
\begin{equation}\label{e:rho}
\rho_{\hh} (X) = \bigoplus_{q\in \Gamma_{\hh}} X(q).
\end{equation} 
Notice that this functor is exact. Denote by $R \rho_{\hh} :\D \to \D^{b} (\Gamma , \Om_{\hh})$ the corresponding derived functor.

\begin{thm}\label{t:equivalence}
Let $\hh$ be a height function, and let $\Gamma_{\hh}$ be the corresponding slice. Then the functor $R \rho_{\hh} :\D \to \D^{b} (\Gamma , \Om_{\hh})$ is an equivalence of triangulated categories.
\end{thm}

\begin{proof}
Note that a height function $\hh$ gives a lifting of the quiver $(\Gamma , \Om_{\hh} )$ to $\Gammahat$ and that the image of $( \Gamma , \Om_{\hh} )$ is the slice $\Gamma_{\hh}$. Hence for any object $Y \in \D^{b} (\Gamma , \Om_{\hh})$ define an object in $\D$ as follows.

For $i \in \Gamma$ and $q=(i, \hh (i)) \in \Gammahat$ define $X(q) = Y(i)$ and for each edge $e:q \to q^{\prime} \in \Gamma_{\hh}$ define maps $x_{e} : X(q) \to X(q^{\prime})$ by $x_{e} = y_{e}$ where $y_{e} : Y(i) \to Y(j)$ and $q^{\prime} = (j, \hh (j))$. Then Part 1 of Lemma~\ref{l:determined} shows this determines an object $X \in \D$. Hence for every $Y \in \D^{b} (\Gamma , \Om_{\hh} )$ there exists $X \in \D$ such that $R \rho_{\hh} (X) = Y$. Note that Part 2 of Lemma~\ref{l:determined} implies that for any $X,Y \in \D$ there is an isomorphism $\Hom_{\D} (X,Y) \simeq \Hom_{\D^{b} (\Gamma , \Om_{\hh})} (R \rho_{\hh} X , R \rho_{\hh} Y)$. Together, this shows that $R \rho_{\hh}$ is an equivalence, and since $R \rho_{\hh}$ is the derived functor of an exact functor it is a triangle functor.
\end{proof}

\begin{cor}\label{c:ARquiver}
Let $\Gamma$ be Dynkin. The Auslander-Reiten quiver of $\D$ is $\Gammahat^{op}$, so the objects $X_{q}$ form a complete list of indecomposable objects in $\D$.
\end{cor}
\begin{proof}
By Theorem~\ref{t:equivalence} the Auslander-Reiten quiver of $\D$ is isomorphic to that of $\D^{b} (\Gamma ,\Om_{\hh})$. It is well known (see \cite{happel} for example) that the Auslander-Reiten quiver of $\D^{b} (\Gamma ,\Om_{\hh})$ is isomorphic to $\Gammahat^{op}$.
\end{proof}

\begin{remark}
Usually the Auslander-Reiten quiver of $\D^{b} (\Gamma , \Om_{\hh})$ is identified with $\Gammahat$ by identifying the projectives with a slice in $\Gammahat$ that gives the opposite orientation to $\Om_{\hh}$, and proceeding from there. For reasons that will become clear in Section~\ref{s:mesh}, we instead identify it with $\Gammahat^{op}$ by identifying the projectives with the slice $\Gamma_{\hh} \subset \Gammahat^{op}$.
\end{remark}

\begin{cor}
The category $\D$ has Serre Duality:
$$\Hom_{\D} (X,Y) =( \Ext_{\D}^{1} (Y, \tau_{\D} X))^{*}$$
where $^{*}$ denotes the dual space and $\tau_{\D}$ is given by Equation~\ref{e:tauD}.
\end{cor}
\begin{proof}
It is well known (see \cite{happel} Proposition 4.10 p.42) that in the category $\D^{b} (\Gamma, \Om_{\hh})$ this relation holds. 
\end{proof}

The following theorem shows that the restriction functor is compatible with the reflection functors.

\begin{thm}\label{t:reflecfunc} Let $i$ be a source (or sink) for the orientation $\Om_{\hh}$ and let $S_{i}^{\pm}$ denote the corresponding reflection functor. Then the following diagram is commutative.
$$\xymatrix{
&  \D^{b} (\Gamma ,\Om_{\hh})   \ar[dd]^{RS^{\pm}_{i}} \\
\D \ar[dr]_{R \rho_{s^{\pm}_{i} \hh}} \ar[ur]^{R \rho_{\hh}} & \\
& \D^{b} (\Gamma ,\Om_{s^{\pm}_{i} \hh}) \\
}$$

\end{thm}

\begin{proof}
Let $q\in \Gamma_{\hh}$ be a source. Let $X \in \D$, so that there is a fixed identification $X(\tau q) \simeq Cone(X(q) \to \oplus X(q^{\prime}))$. Restriction along the slice $\Gamma_{s^{+}_{i} \hh}$ gives $X(p)$ if $p\neq q$ and $X(\tau q)$ if $p=q$. The reflection functor is defined as $X(p)$ if $p\neq q$ and $Cone(X(q) \to \oplus X(q^{\prime}))$ if $p=q$, so the diagram commutes.
\end{proof}

%%%%%%%%%%%%%%%%%%%%%%%%%%%%%%%%%%%%%%%%%%%%%%%%%%%%%%%%%%%
\section{The Mesh Category $\B$}\label{s:mesh}
In the remainder of this paper, fix $\Gamma$ to be Dynkin.
In this section the definition of the mesh category of a translation quiver $(Q, \tau)$ is recalled and then related to the category $\D$.

Let $(Q, \tau)$ be a translation quiver (see \cite{ars} Chapter VII for details). Define the set of indecomposable objects of the mesh category $\B (Q)$ to be the vertices of $Q$. Set $\Hom_{\B} (q,q^{\prime}) = \Path (q,q^{\prime}) / J$ where $J$ is the mesh ideal generated by the mesh relations $\sum_{s(e)=i} \epsilon (e) \bar e e$ (see Equation~\ref{e:mesh}).

Consider the mesh category of the translation quiver $(\Gammahat^{op} , \tau_{\Gammahat} )$ where $\tau_{\Gammahat} (i,n) = (i,n+2)$. For simplicity denote this by $\B$ and denote the translation by $\tau_{\B}$.
\begin{remark}
Note that we consider the mesh category of $\Gammahat^{op}$ instead of $\Gammahat$ since the Auslander-Reiten quiver of $\D$ is $\Gammahat^{op}$.
\end{remark}

It is shown in \cite{bbk} (Section 6) that there are the following automorphisms in $\B$:

\begin{enumerate}
\item A Nakayama automorphism $\nu_{\B}$ which commutes with $\tau_{\B}$. (Here $\nu = \tilde{\beta}^{-1}$ in the notation of \cite{bbk}.)
\item An automorphism $\gamma_{\B}$ defined by $\gamma_{\B} := \nu_{\B} \tau_{\B}^{-1}$.
\end{enumerate}

These automorphisms satisfy:
\begin{equation}\label{e:nugamma}
\begin{aligned}
\nu_{\B}^{2} &= \tau_{\B}^{-(h-2)} \\
\gamma_{\B}^{2} &= \tau_{\B}^{-h}
\end{aligned}
\end{equation}

As before, for any $i\in \Gamma$ define $\check{\imath}$ by $-\alpha_{i}^{\Pi} = w_{0}^{\Pi} (\alpha_{\check{\imath}}^{\Pi})$, where $w_{0}^{\Pi} \in W$ is the longest element and $\Pi$ is any set of simple roots.

In terms of $\Gammahat^{op}$ the maps $\nu_{\B}$ and $\gamma_{\B}$ are given by: 
\begin{equation}
\begin{aligned}
\nu_{\B} (i,n) &= (\check{\imath} \ ,n-h+2) \\
\gamma_{\B} (i,n) &= (\check{\imath} \ ,n-h)
\end{aligned}
\end{equation}

\begin{example}
For the graph $\Gamma = A_{4}$, $\check{\imath} \ = 5-i$ and $h=5$ so $\nu_{\B} (i,n) = (5-i , n-3)$ and $\gamma_{\B} (i,n) = (5-i , n-5)$. The maps $\nu_{\B}$ and $\gamma_{\B}$ are shown in Figure~\ref{f:nugammaop}. 

\begin{figure}[ht]
\centering
\begin{overpic}
{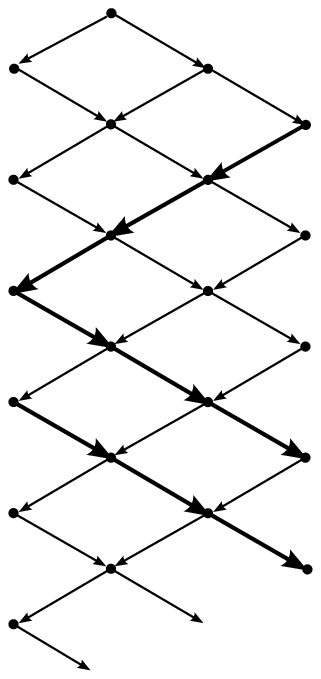}%[grid]
\put (47,80){$\Gamma_{h}$}
\put (47,30){$\nu_{\B} \Gamma_{h}$}
\put (47,13){$\gamma_{\B} \Gamma_{h}$}
\end{overpic}

\caption{The maps $\nu_{\B}$ and $\gamma_{\B}$ in the case $\Gamma = A_{4}$. A slice $\Gamma_{h}$ and its images under $\nu_{\B}$ and $\gamma_{\B}$ are shown in bold. Recall that we are considering the mesh category of $\Gammahat^{op}$ as mentioned above.}\label{f:nugammaop}
\end{figure}
\end{example}

Recall that the Auslander-Reiten quiver of $\D$ is $\Gammahat^{op}$. This identification is given by $[X_{q}] \mapsto q$. In terms of arrows, by Theorem~\ref{t:RHom} Part 2, for each arrow $q \to q^{\prime}$ in $\Gammahat$ there is an arrow $q^{\prime} \to q$ in the Auslander-Reiten quiver.

In the category $\D$ define an automorphism $\tau_{\D}$ by 
\begin{equation}\label{e:tauD}
\tau_{\D} (X) (q^{\prime}) = X (\tau^{-1} q^{\prime}).
\end{equation}
Notice that $$X_{\tau q} (q^{\prime}) = \Path_{\Gammahat}^{\bullet} (\tau q , q^{\prime} ) \simeq \Path_{\Gammahat}^{\bullet} (q, \tau^{-1} q^{\prime} )= \tau_{\D} (X_{q}) (q^{\prime})$$ so that $\tau_{\D} X_{q} \simeq X_{\tau q}$ for the indecomposables $X_{q}$. In terms of the Auslander-Reiten quiver of $\D$, this identifies $\tau_{\D}$ with the translation $\tau$ on $\Gammahat^{op}$.

\begin{thm}\label{t:mesh}
Let $\hh$ be a height function.
\begin{enumerate}
\item There are equivalences of additive categories, given by the following commutative diagram:
$$\xymatrix{
\B & \D^{b} (\Gamma ,\Om_{\hh}) \ar[l]_{\Psi_{\hh}} \\
& \D \ar[u]_{R \rho_{\hh}} \ar[ul]^{\psi} \\
}$$
\item Under these equivalences the automorphisms $\nu_{\B}$ gives an Nakayama automorphism $\nu_{\D}$ on $\D$ and is identified with the Nakayama automorphism $\nu$ in $\D^{b} (\Gamma, \Om_{\hh})$
\item The map $\tau_{\B}$ can be identified with the Auslander-Reiten translation in $\D^{b} (\Gamma ,\Om_{\hh})$, and with $\tau_{\D}$ in $\D$.
\item The map $\gamma_{\B}$ can be identified with $T$ in $\D^{b} (\Gamma, \Om_{\hh})$, and with $T_{\D}$ in $\D$. Hence we can impose a triangulated structure on $\B$ making the equivalences in (1) triangulated equivalences.
\item In $\D$ the relation $T^{2} = \tau_{\D}^{-h}$ holds and hence the objects $X_{q}$ satisfy the relation $X_{q}^{\bullet +2} \simeq X_{\tau^{-h} q}^{\bullet}$.
\end{enumerate}
\end{thm}

\begin{proof}
\par\indent
\begin{enumerate}
\item The equivalence $\rho_{\hh}$ is from Theorem~\ref{t:equivalence}. \\
The equivalence $\Psi_{\hh}$ is  the map which is given $P_{i} \mapsto (i, \hh (i))$ on projectives, so that the projectives in $\Rep (\Gamma , \Om_{\hh})$ map to the slice $\Gamma_{\hh} \subset \Gammahat^{op}$. (Note that in $\Gammahat^{op}$ the arrows are reversed, so this agrees with the usual identification of projectives with a slice giving the orientation opposite to $\Om_{\hh}$.) This is just the identification of $\Ind (\D^{b}(\Gamma, \Om_{\hh}))$ with its Auslander-Reiten quiver. That this is an equivalence is well-known, see \cite{happel} for example.\\
The equivalence $\psi$ is given by $X_{q} \mapsto q$. By Corollary~\ref{c:ARquiver} the objects $X_{q}$ form a complete list of indecomposables, and since 
\begin{align*}
\Hom_{\D} (X_{q}, X_{q^{\prime}}) &= H^{0} (X_{q^{\prime}} (q)) = \Path_{\Gammahat} (q^{\prime}, q)/J = \Path_{\Gammahat^{op}} (q, q^{\prime})/J \\
&= \Hom_{\B} (q, q^{\prime})
\end{align*}
it follows that this is an equivalence.
\item Follows from (1), since $\Psi_{\hh}$ is the identification of $\Ind (\D^{b} (\Gamma ,\Om_{\hh}))$ with its Auslander-Reiten quiver.
\item Again follows from (1).
\item In $\D^{b} (\Gamma ,\Om_{\hh})$ the Auslander-Reiten translation is defined by $\tau_{\D^{b}} := T^{-1} \nu$, or equivalently $T = \nu \tau^{-1}$
\item By (4) $\gamma_{\B}$ is identified with $T_{\D}$, by (2) and (3) so there is an identification of Nakayama automorphisms and translations in $\B$ and $\D$. Then using Equation~\ref{e:nugamma} gives the result.
\end{enumerate}
\end{proof}

%%%%%%%%%%%%%%%%%%%%%%%%%%%%%%%%%%%%%%%%%%%%%%%%%%%%%%%%%%%
\section{Periodicity in the Dynkin Case}\label{s:period}

This section discusses periodicity of the ``dg-preprojective algebra" $A$ in the case where $\Gamma$ is Dynkin.
Recall the decomposition given in Section 2:
\begin{equation}
A=\bigoplus_{n\le 0}A^n=\bigoplus_{i,j\in \Gamma , \ l\in \Z_+, n\le 0}
A^n_{i,j;l}
\end{equation}
Note that the differential in $A$ preserves the grading by path length, so this decomposition passes to homology:
$$H^{n} (A) = \bigoplus_{i,j;l} H^{n}(A_{i,j;l})$$

Now fix $q=(j,n+l)$ and $q^{\prime} = (i,n)$. Then by definition of $X_{q^{\prime}}$, and by Part 2 of Theorem~\ref{t:RHom}, $\RHom (X_{q} , X_{q^{\prime}}) = X_{q^{\prime}}(q) = A_{i,j;l}$ so the component $A_{i,j;l}$ can be interpreted as the $\RHom$ complex of the corresponding indecomposables.

Recall that in the case where $\Gamma$ was not Dynkin the complex $A$ was a dg-resolution of the preprojective algebra $\Pi$. In particular, all homology was in degree 0. The decomposition of homology above and the identification $A_{i,j;l} \simeq \RHom_{\D} (X_{q}, X_{q^{\prime}})$ makes it clear that this is not the case when $\Gamma$ is Dynkin. 

In the case where $\Gamma$ is Dynkin there is the following periodicity result for the Koszul complex of the preprojective algebra and its homology. This is likely known to experts, but does not seem to be easily available in the literature.

\begin{thm}\label{t:periodic}
There is a quasi-isomorphism of complexes
\begin{equation}\label{e:periodicity}
A_{i,j;l}^{\bullet +2} \simeq A_{i,j;l+2h}^{\bullet}.
\end{equation}
In terms of the homology of the complex $A$ this gives:
$$H^{k+2}(A_{i,j;l}) = H^{k}(A_{i,j;l+2h})$$
\end{thm}

\begin{proof}
By Theorem~\ref{t:mesh} Part 5 there is an identification $X_{q^{\prime}}^{\bullet +2} \simeq X_{\tau^{-h} q^{\prime}}^{\bullet}$ in $\D$, so in particular  $X_{q^{\prime}}^{\bullet +2} (q) \simeq X_{\tau^{-h} q^{\prime}}^{\bullet}(q)$. For $q=(i,n)$ and $q^{\prime} =(j,n+l)$  there are identifications 
\begin{equation}\label{e:periodic}
\begin{aligned}
& \RHom(X_{q} , X_{q^{\prime}}) = X_{q^{\prime}} (q) = A_{i,j;l} \\ 
& \RHom (X_{q} , X_{\tau^{-h} q^{\prime}}) = X_{\tau^{-h} q^{\prime}} (q) = A_{i,j;l+2h}. 
\end{aligned}
\end{equation}

Combining these gives $A_{i,j;l}^{\bullet +2} \simeq A_{i,j;l+2h}$.
\end{proof}

In terms of the category $\D$ this result can be interpreted as follows.

\begin{cor} Let $X,Y \in \D$.
\begin{enumerate}
\item $\RHom^{\bullet} (X, \tau^{-h} Y) \simeq \RHom^{\bullet +2} (X , Y)$
\item $\Ext^{i} (X, \tau^{-h} Y) \simeq \Ext^{i+2} (X, Y)$
\end{enumerate}
\end{cor}

\begin{proof} First note that the objects $X_{q}$ form a complete list of indecomposables in $\D$ (in this section $\Gamma$ is Dynkin). So it's enough to prove this result for these objects. For these objects, recalling that $\tau X_{q} \simeq X_{\tau q}$ shows that the first statement follows from Equation~\ref{e:periodic} in the proof of Theorem~\ref{t:periodic}. The second statement follows by taking homology.
\end{proof}

%%%%%%%%%%%%%%%%%%%%%%%%%%%%%%%%%%%%%%%%%%%%%%%%%%%%%%%%%%%
\section{The quotient category $\D /T^{2}$}\label{s:quotcat}
It was shown in \cite{px} that the category $\D^{b} (\Gamma , \Om)/ T^{2}$ is a triangulated category, and that the set $\Ind(\Gamma , \Om)$, of classes of indecomposables gives the corresponding root system. More general quotient categories were studied in \cite{keller}, where conditions for a quotient category to inherit a triangulated structure are given.

In this section we consider the quotient category $\C = \D / T^{2}$ and relate it to Theorem~\ref{t:main1}.

\begin{prop}\label{p:quotient}
The quotient category $\D /T^{2}$ has the following properties:
\begin{enumerate}
\item It is triangulated.
\item $\tau^{-h} = Id = \tau^{h}$
\item It has Auslander-Reiten quiver $\Gammahat / \tau^{h} = \Gammahat_{cyc}$.
\end{enumerate}
\end{prop}

\begin{proof}
Part 1 follows from the main result of \cite{keller}. The other parts follow from Theorem~\ref{t:mesh}.
\end{proof}

Let $R$ be the root system corresponding to $\Gamma$. Proposition~\ref{p:quotient}, the equivalences in Section~\ref{s:equiv} and the results of \cite{kt} show that there is a bijection between $R$ and the Auslander-Reiten quiver of the category $\C = \D / T^{2}$. The following Theorem summarizes this bijection and completes the proof of Theorem~\ref{t:main1}. 

\begin{thm}
Let $\Gamma$ be Dynkin with root system $R$ and let $\K$ be the Grothendieck group of $\C$. Set $\< X,Y \> = \dim \RHom (X,Y) = \dim \Hom (X,Y) - \dim \Ext^{1} (X,Y)$. The set $\Ind \subset \K$ of isomorphism classes of indecomposable objects in $\C$, with bilinear form given by $(X,Y) = \< X,Y \> + \< Y, X \>$, is isomorphic to $R$, and $\K$ is isomorphic to the root lattice. Moreover, the translation $\tau$ gives a Coxeter element for this root system, and $T_{\C}$ gives the longest element.
\end{thm}

%%%%%%%%%%%%%%%%%%%%%%%%%%%%%%%%%%%%%%%%%%%%%%%%%%%%%%%%%%%
\bibliographystyle{amsalpha}

%%%%%%%%%%%%%%%%%%%%%%%%%%%%%%%%%%%%%%%%%%%%%%%%%%%%%%%%%%
\end{document}